\documentclass[11pt,twoside]{article}
\usepackage{amsmath, amssymb, amsfonts, amstext, amsthm, textcomp, enumerate}
\usepackage[mathscr]{euscript}
\usepackage{float}
\usepackage{booktabs}
\usepackage{mathtools}
\usepackage{graphicx}
\usepackage{caption}
\usepackage{epstopdf}
\usepackage{longtable}
\usepackage[utf8]{inputenc}
\usepackage{color}
\usepackage{hyperref}
\usepackage{graphicx}
\usepackage{dcolumn}% Align table columns on decimal point
\usepackage{bm}% bold math
\usepackage{epstopdf}
\usepackage[english]{babel}
\usepackage{subfigure}
\usepackage{color}
\usepackage{ulem}
\usepackage{comment}

\newcommand{\TR}{\textcolor{red}}

\newtheorem{thm}{Theorem}[section]
\newtheorem{lem}[thm]{Lemma}%[section]
\newtheorem{cor}[thm]{Corollary}%[section]
\newtheorem{pro}[thm]{Proposition}%[section]
\newtheorem{defn}[thm]{Definition}%[section]
\newtheorem{obs}[thm]{Observation}
\newtheorem{rem}[thm]{Remark}%[section]
\newtheorem{ex}[thm]{Example}%[section]

\newcommand{\RR}{\mathbb{R}}

\newcommand{\mnr}{\mathbf{M}_n\,(\mathbb{R})}

\newfont{\bb}{msbm10}

\def\1{{\sf 1}}
\def\0{{\sf 0}}

\def\qed{\hfill {\hbox{\footnotesize{$\Box$}}}}

\def\RR{\mathbb{R}}

\def\qed{\hfill {\hbox{\footnotesize{$\Box$}}}}
\def\0{{\sf {\sf 0}}}

\def\0{{\sf 0}}

\def\proof{{\noindent\bf Proof.}\hskip 0.3truecm}
\def\BBox{\kern  -0.2cm\hbox{\vrule width 0.15cm height 0.3cm}}

\def\S{\mathcal{S}}

\def\RR{\mathbb{R}}

\begin{document}
\begin{center} 
\begin{large} {\bf On an Analogue of a Property of Singular $M$-matrices, for the Lyapunov and the Stein Operators}
\end{large}	
\end{center}
		\begin{center}
		
			\it{A.M. Encinas $^1$, Samir Mondal $^2$ and K.C. Sivakumar $^2$
			}\\
   \end{center}
\footnotetext[1]{Department of Mathematics, Polytechnic University of Catalunya, Barcelona, Spain (andres.marcos.encinas@upc.edu).}

   \footnotetext[2]{Department of Mathematics, Indian Institute of Technology Madras,
			Chennai 600036, India (ma19d750@smail.iitm.ac.in, kcskumar@iitm.ac.in).}

		%%% ----------------------------------------------------------------------
		\begin{abstract}
			In the setting of real square matrices, it is known that, if $A$ is a singular irreducible $M$-matrix, then the only nonnegative vector that belongs to the range space of $A$ is the zero vector. In this paper, we prove an analogue of this result for the Lyapunov and the Stein operators.
		\end{abstract}
		%%% ----------------------------------------------------------------------
		
		
		%\vskip.05in
		%90C33. %Linear Complementarity
		
\newpage
	
\section{Introduction}

The set of all real matrices of order $n \times n$ will be denoted by $\mnr$. For a matrix $X \in \mnr$, we use $X\geq 0$ to denote the fact that all the entries of $X$ are nonnegative. If all the entries of $X$ are positive, we denote that by $X >0$. We use the same notation for vectors.

As usual, for any $A\in \mnr$, $N(A)$ and $R(A)$ denote the null and the range spaces of $A\in \mnr$, respectively. The  {\it rank} of $A$ is ${\rm rk}(A)$, the dimension of $R(A)$. In addition, $\rho(A)$ denotes the {\it spectral radius} of $A$; that is,  the maximum of the absolute values of its eigenvalues.

A matrix $A\in\mnr$ is said to be {\it reducible} if there
is a permutation matrix $P\in\mnr$ such that $P^{T}AP$ has
the form
\[%
\begin{bmatrix}
S_{11} & S_{12}\\
0 & S_{22}%
\end{bmatrix}
\]
for some square matrices $S_{11}$ and $S_{22}$ of order at least one. A matrix
is {\it irreducible} if it is not reducible. 

The {\it group inverse} of a matrix $A\in \mnr$ is the unique matrix $X \in \mnr$, if it exists, that satisfies the equations $AXA=A, XAX=X$ and $AX=XA$. If it exists, then the group inverse is denoted by $A^{\#}$. Of course, when $A$ is nonsingular, then $A^\#$ exists and moreover $A^\#=A^{-1}$.

A necessary and sufficient condition for the group inverse of a matrix $A$ to exist is that $A$ has index $1$; that is, $R(A^2)=R(A)$, which is equivalent to the condition $N(A^2)=N(A)$, see \cite[Theorem 2, Section 4.4]{bg}. Another characterization is that $A^{\#}$ exists iff $R(A)$ and $N(A)$ are complementary subspaces of $\mathbb{R}^n$. It is easy to show that any symmetric matrix has group inverse, that a nilpotent matrix $A$ (viz., $A^n=0$), does not have group inverse, whereas the group inverse of an idempotent matrix (viz., $A^2=A$), (exists and) is itself. We refer the reader to \cite[Chapter 4]{bg} for more details.

If ${\cal S}^n(\RR)$ is the subspace of symmetric matrices in $\mnr$, for $X \in {\cal S}^n(\RR),$ let us signify $X \succeq 0$ to denote the fact that $X$ is a positive semidefinite matrix. We use $X \succ 0,$ when $X$ is positive definite.

\begin{defn}
A matrix $A\in \mnr$ is called a $Z$-matrix, if all the off-diagonal entries of $A$ are nonpositive. Any $Z$-matrix $A$ has the representation $A=sI-B,$ where $s \geq 0$ and $B\ge 0$. If $s\ge \rho(B)$, then $A$ is called an {\it $M$-matrix}. 
\end{defn}

The notion of $M$-matrix  introduced by A. Ostrowski in 1935 in honor of H. Minkowski who worked with this class of matrices around 1900. $M$-matrices possess many interesting nonnegativity properties. For instance, in the representation as above, let $s > \rho(B).$ Then, $A=s\big(I-\frac{B}{s}\big)$ and considering the {\it Neumann series} for $\frac{B}{s}$; that is,
$$\Big(I-\frac{B}{s}\Big)^{-1}=\sum_{m=0}^{\infty} \Big(\frac{B}{s}\Big)^m \geq 0.$$ 
we conclude that $A$ is invertible and moreover $A^{-1}=\frac{1}{s}\Big(I-\frac{B}{s}\Big)^{-1}\ge 0$. In such a case, we shall refer to $A$ as an {\it invertible $M$-matrix}. Otherwise, we call $A$, a {\it singular $M$-matrix}. 

Here is a characterization for invertible $M$-matrices. For a proof, we refer the reader to the book \cite[Chapter 6]{berpl}, where fifty different characterizations are given.  

\begin{thm}\cite[Theorem 2.3, Chapter 6]{berpl}\label{matrixcase} \leavevmode  \\
Let $A$ be a $Z$-matrix. Then the following statements are equivalent:\\
(a) $A$ is an invertible $M$-matrix.\\
(b) There exists $x>0$ such that $Ax>0$.\\
(c) For every $q >0,$ there exiss $x>0$ such that $Ax=q.$\\
(d) $A$ is monotone, i.e. $Ax \geq 0 \Longrightarrow x\geq 0$.\\ 
$A$ is inverse-nonnegative, i.e. $A$ is invertible and $A^{-1}\ge 0$. \\
(f) $A$ is a $P$-matrix, i.e. all principal minors of $A$ are positive.\\
(g) $A$ is positive stable i.e. if the real part of each of its eigenvalues is positive. \\
Suppose that $A$ has the representation $A=sI-B$, with $B \geq 0$ and $s\ge \rho(B)$. 
Then each of the above statements is equivalent to:\\
(h) $s<\rho(B)$.\\
When in addition, $A$ is irreducible the above statements are equivalent to:\\
(e') $A$ is inverse-positive, i.e. $A$ is invertible and $A^{-1}> 0$.
\end{thm}

We will have the occasion to consider versions of Theorem \ref{matrixcase} for two special classes of operators on $S^n(\RR)$ (Theorem \ref{lyapunovcase} and Theorem \ref{steincase}), in the next section.

Next, let us turn our attention to the case singular $M$-matrices. A distinguished subclass of such matrices, due to their relevance in many applications, is the set of {\it singular irreducible $M$-matrices}. For such matrices, we us recall a well known result, which will serve to motivate the contents of this article.

\begin{thm}\label{berplthm}\cite[Theorem 4.16, Chapter 6]{berpl},\cite[Theorem 3]{plem}\leavevmode \\
Let $A \in \mnr$ be a singular irreducible $M$-matrix. Then the following hold:\\
(a) $\rm rk(A)=n-1.$\\
(b) There exists a vector $x >0$ such that $Ax=0.$\\
(c) $A^{\#}$ exists and is nonnegative on $R(A)$, i.e. $x \geq 0, x \in R(A) \Longrightarrow A^{\#}x \geq 0$.\\
(d) All the principal submatrices of $A$, except $A$ itself is an invertible $M$-matrix. \\
(e) $A$ is almost monotone, i.e. $Ax \geq 0 \Longrightarrow Ax=0.$
\end{thm}

We remark that in \cite[Theorem 4.16, Chapter 6]{berpl} statement (c) was established as $A$ has "property c", and that the equivalence with (c) in Theorem \ref{berplthm} was proved in \cite[Theorem 3]{plem}. Recall that a singular $M$-matrix $A$ is said to have "property c", if $A=sI-B$, where $B \geq 0, s>0$ and the matrix $\frac{1}{s}B$ is semi-convergent. Semi-convergence of a matrix $X$ means that the matrix sequence $\{X^k\}$ converges. Note that for an invertible $M$-matrix, such a sequence converges to zero. One of the most prominent situations where such matrices arise concerns irreducible Markov processes, where one deals with matrices of the form $I-T$, where $T$ is an irreducible column stochastic matrix. 
We refer the reader to \cite{li} for additional results that characterize singular irreducible $M$-matrices.

\begin{defn}
A matrix $A$ is referred to as {\it range monotone (see, for instance \cite{miskcs}}), if 
$$Ax \geq 0, x \in R(A) \Longrightarrow x \geq 0.$$ If $A$ is range monotone, then we say that $A$ has the {\it range monotonicity} property.
\end{defn}

Observe that range monotonicity of $A$ is equivalent to: $A^2x \geq 0 \Longrightarrow Ax\ge 0.$ Moreover, \cite[Theorem 3 (e)]{plem} shows that  statement $(c)$ of Theorem \ref{berplthm} is equivalent to the range monotonicity of $A$.

Let us recall a notion that is stronger than range monotonicity. This will play a central role in this article.

\begin{defn}
Let $A \in \mnr$. Then $A$ is called trivially range monotone if 
$$Ax \geq 0, x \in R(A) \Longrightarrow x = 0.$$
\end{defn}

Trivial range monotonicity of $A$ is the same as: $A^2x \geq 0 \Longrightarrow Ax=0.$ It is easy to verify that the matrix $A=\begin{pmatrix}
~~1 & -1 \\
-1 & ~~1
\end{pmatrix}$ is trivially range monotone. 

\begin{thm}\label{trivrangemon}
Let $A \in \mnr$ be a singular irreducible $M$-matrix. Then $A$ is trivially range monotone.
\end{thm}
\proof
Suppose that $Ax \geq 0, x \in R(A).$ Then by the almost monotonicity of $A$ (item (e), Theorem \ref{berplthm}), it follows that $Ax=0$. So, $x \in N(A)\cap R(A)$ and since the group inverse $A^{\#}$ exists, we have $x=0.$ 
\qed

In this work, we shall be interested in the problem of determining trivial range monotonicity for the Lyapunov and the Stein operators on the space of real symmetric matrices. We consider four classes of matrices that give rise to specific instances of these operators, for which we give affirmative/negative answers. The main results are presented in Theorem \ref{idem_result}, Theorem \ref{gpinvexist} and Theorem \ref{trangemonotone}. Pertinent counter\-examples are presented in Section \ref{cex}, followed by a summary table consolidating the findings.

\section{Preliminaries}
A nonempty subset $C$ of a finite dimensional Hilbert space $H$ is said to be a {\it proper cone} if $C+C=C, \alpha C \subseteq C$, for all $\alpha \geq 0$, $C \cap -C =\{0\}$ and $C$ has a nonempty interior. Given a cone $C$, we define its {\it dual cone} $C^*$ by 
%$$\hbox{\textst{C^*:=\{ y \in \mathbb{R}^n: x^Ty \geq 0, ~\forall x \in C\}.}}$$
$$ C^*:=\{ y \in \mathbb{R}^n: \langle x,y\rangle \geq 0, ~\forall x \in C\}.$$
It follows that $C^*$ is indeed a proper cone. A well known example of a proper cone is the nonnegative orthant $\mathbb{R}^n_+$, viz., the cone of nonnegative vectors. Let the space of all real symmetric matrices of order $n,$ be denoted by ${\cal S}^n$. Then ${\cal S}^n$ is a real Hilbert space with the trace inner product, i.e., $\langle A,B \rangle = tr(AB),~A,B \in {\cal S}^n.$ In ${\cal S}^n,$ the set of all positive semidefinite matrices denoted by ${\cal S}^n_+$, is a cone. Both these cones satisfy the condition of self-duality, viz., $C^*=C,$ when $C=\mathbb{R}^n_+$ or ${\cal S}^n_+$.

Next, let us observe that if $A$  is a $Z$-matrix, then $\langle Ae_j,e_i \rangle = a_{ij} \leq 0 ,\;i\neq j$, where $e_i$ denotes the $i$th standard basis vector of $\mathbb{R}^n$. These basis vectors are mutually orthogonal and positive, as well. It then follows that an equivalent formulation for a matrix $A$ to be a $Z$-matrix is:
\begin{center}
$x \geq 0, \,y \geq 0$\; and \;$\langle x,y \rangle =0 \;\Longrightarrow \; \langle Ax,y \rangle \leq 0$.
\end{center}
This in turn is the same as saying (with $C=\mathbb{R}^n_+$):
\begin{center}
$x \in C, \,y \in C^*$ \; and \;$\langle x,y \rangle =0 \; \Longrightarrow \; \langle Ax,y \rangle \leq 0$.
\end{center}

Motivated by the reformulation above, a linear operator $T:{\cal S}^n \longrightarrow {\cal S}^n$ is called a {\it $Z$-operator} if it satisfies: 
\begin{center}
$X \succeq 0, Y \succeq 0$ \; and \;$\langle X,Y \rangle =0 \; \Longrightarrow \; \langle T(X),Y \rangle \leq 0$,
\end{center}
where $U \succeq 0$ stands for the fact that the symmetric matrix $U$ is positive semidefinite. When $-U \succeq 0,$ we will use the notation $U \preceq 0.$
In this article, we shall be concerned with two important classes of operators on ${\cal S}^n$. These are the Lyapunov operator and the Stein operator. Let us first look at their definitions.
Given $A \in \mnr$, the {\it Lyapunov operator} $L_A$ on ${\cal S}^n$ is the operator
  \[  L_A(X):=AX+XA^T, \;X \in {\cal S}^n \]
and the {\it Stein operator} $S_A$ is defined by
  \[S_A(X):=X-AXA^T, \;X \in {\cal S}^n.\]

Let $L$ denote either the Lyapunov or the Stein operator for a given matrix $A$. Then $L$ satisfies the following (independent of $A$):
\begin{center}
$X \succeq 0,\, Y \succeq 0$ \; and \; $\langle X,Y \rangle =0 \; \Longrightarrow \; \langle L(X),Y \rangle \leq  0.$
\end{center}
Thus, both these operators could be thought of as analogues of $Z$-matrices, for linear maps on ${\cal S}^n$. 

An operator $T$ on ${\cal S}^n$ will be called a {\it positive stable $Z$-operator} if $T$ is an invertible $Z$-operator and $T^{-1}({\cal S}^n_+) \subseteq {\cal S}^n_+$. Note that this generalizes what we know for an invertible $M$-matrix, namely that it is invertible and that its inverse is nonnegative, i.e. leaves the cone $\mathbb{R}^n_+$ invariant.

The following result is a version of (items (b) and (c) of) Theorem \ref{matrixcase} for the Lyapunov operator \cite{gowsong}. The notation $U \succ 0$ denotes the fact that the symmetric matrix $U$ is positive definite. 

\begin{thm} \label{lyapunovcase} (\cite[Theorem 5]{gowsong}) \leavevmode  \\
For $A \in \mnr$, the following statements are equivalent:\\
(a) There exists $E \succ 0$ such that $L_A(E) \succ 0$.\\
(b) For every $Q \succ 0$ there exists $X \succ 0$ such that $L_A(X)=Q$.\\
(c) $A$ is positive stable.
\end{thm}

A version for the Stein operator is stated next. The original result holds for complex matrices \cite{gowtp}. We present only the real case. Recall that a matrix $A$ is called Schur stable if all its eigenvalues lie in the open unit disc of the complex plane.

\begin{thm} \label{steincase} (\cite[Theorem 11]{gowtp})\leavevmode  \\
For $A \in \mnr$, the following statements are equivalent:\\
(a) There exists $E \succ 0$ such that $S_A(E) \succ0$.\\
(b) For every $Q \succ 0$ there exists $X \succ 0$ such that $S_A(X)=Q$.\\
(c) $A$ is Schur stable.
\end{thm}

These results may be considered as analogues of Theorem \ref{matrixcase}. Let $L$ stand for either the Lyapunov operator or the Stein operator. Then, $(b)$ of Theorem \ref{lyapunovcase} or Theorem \ref{steincase} states that the range of the operator $L$ contains the open set of all (symmetric) positive definite matrices. Hence $L$ is surjective and so, it is injective, too. Thus, $L$ is invertible. By the same statement, it also follows that $L^{-1}({\cal S}^n_+) \subseteq {\cal S}^n_+$. Thus, these two operators are examples of positive stable $Z$-operators.

\section{Main Results}
We are interested in identifying some matrix classes for which the Lyapunov operator and/or the Stein operator are/is trivially range monotone. Let us define this notion, first. Motivated by the definitions for matrices, we call a linear operator $T:{\cal S}^n \rightarrow {\cal S}^n$ {\it range monotone} if $$T(X) \succeq 0, X \in R(T) \Longrightarrow X \succeq 0$$ and refer to $T$ as {\it trivially range monotone} if $$T(X) \succeq 0, X \in R(T) \Longrightarrow X =0. $$
As in the case of matrices, one may observe that range monotonicity of a linear operator $T$ is equivalent to the condition: $T^2(X) \succeq 0 \Longrightarrow T(X) \succeq 0$ and that trivial range monotonicity is the same as: $T^2(X) \succeq 0 \Longrightarrow T(X) =0.$

Next, we give an example of a range monotone operator and also one which is not range monotone. 

\begin{ex}
Let $A=
\begin{pmatrix}
1 & 1 \\
0 & 0
\end{pmatrix}.$ The corresponding Lyapunov operator is given by $$L_A(X)=
\begin{pmatrix}
2(x_{11}+x_{12}) & x_{12}+x_{22} \\
x_{12}+x_{22} & 0
\end{pmatrix},$$
where $X= \begin{pmatrix}
x_{11} & x_{12} \\
x_{12} & x_{22}
\end{pmatrix}.$ The choice $x_{11}=x_{22}=1, x_{12}=-1$ shows that $L_A$ is not invertible. It is easily seen that if $X \in R(L_A)$, then $x_{22}=0.$ Next, for such an $X$, if $L_A(X)$ is a positive semidefinite matrix, then $X= \begin{pmatrix}
x_{11} & 0 \\
0 & 0
\end{pmatrix}$ is also positive semidefinite. Thus, $L_A$ is a range monotone operator. It is not trivially range monotone, as $X= \begin{pmatrix}
1 & 0 \\
0 & 0
\end{pmatrix} \in R(L_A) \cap {\cal S}^n_+.$
\end{ex}

\begin{ex}\label{remstein}
Let $A=
\begin{pmatrix}
1 & 1 \\
0 & 1
\end{pmatrix}.$ Then the associated Stein operator is given by 
$$S_A(X)=-
\begin{pmatrix}
2b +c & c \\
c & 0
\end{pmatrix},$$
given $X=
\begin{pmatrix}
a & b \\
b & c
\end{pmatrix}.$
Set $X=\begin{pmatrix}
-1 & 0 \\
~~0 & 0
\end{pmatrix}$ and $Y=\begin{pmatrix}
0 & \frac{1}{2} \\
\frac{1}{2} & 0
\end{pmatrix}.$ Then $S_A(Y)=X,$ so that $X \in R(S_A)$. Further, $S_A(X)=0$. This shows that $S_A$ is not range monotone. 
\end{ex}

It is easy to prove that if an operator is idempotent, then it is range monotone. The only fact that is used, is that an idempotent operator acts like the identity operator on its range space. Let us record this statement. 

\begin{lem}\label{idemrangemonotone}
Let $T:V \rightarrow V$ be idempotent. Then $T$ is range monotone.
\end{lem}

As a first step, we address the question of when the Lyapunov and the Stein operators are idempotent. The relevant results are Theorem \ref{Ly:ch} and Theorem \ref{St:ch}. It appears that this question has not been addressed, to the best of our knowledge.

Let us fix some more notation. The all ones vector in $\RR^n$ is denoted by $e$. For any $j=1,\ldots,n$, we define the symmetric matrix $E_{jj}=[0,\ldots,\stackrel{j\atop \downarrow}{e_j},\ldots,0]$ and for any $1\le i<j\le n$ we also define the symmetric matrix $E_{ij}=[0,\ldots,\stackrel{i\atop \downarrow}{e_j},\ldots,\stackrel{j\atop \downarrow}{e_i},\ldots,0]$. It is clear that $\{E_{ij}\}_{1\le i\le j\le n}$ is a basis of $\S^n$. Thus, if $A=[c_1,\ldots,c_n]\in \mnr$, we have,  $AE_{jj}=[0,\ldots,\stackrel{j\atop \downarrow}{c_j},\ldots,0]$, for any $j$ whereas, $AE_{ij}=[0,\ldots,\stackrel{i\atop \downarrow}{c_j},\ldots,\stackrel{j\atop \downarrow}{c_i},\ldots,0]$, for any $1\le i<j\le n$.\\

In what follows, we collect some properties of the Lyapunov operator. We only prove the second item, as the others are straightforward. 

\begin{lem}
\label{Ly:properties1}\leavevmode \\
(a) For any $A,B\in \mnr$ and any $s,t\in \RR$,  $L_{tA+sB}=tL_A+sL_B$.\\
(b) $L_A$ is the zero operator iff $A=0$ and $L_A$ is the identity operator iff $A=\frac{1}{2}I$.\\
(c) For any $A\in \mnr$, $L_A(I)=A+A^T$ and hence $L_A(A+A^T)=L_A^2(I).$\\
(d) For any orthogonal matrix $P$, we have: $L_{PAP^T}(PXP^T)=PL_A(X)P^T.$
\end{lem}
\proof
(b) Let $c_j=(c_{1j},\ldots,c_{nj})^T$ be the $j$-th column of $A$, for $1 \leq j \leq n.$ Then 
$$L_A(E_{jj})=[0,\ldots,c_j,\ldots,0]+[0,\ldots,c_j,\ldots,0]^T.$$
Therefore, when $L_A=0,$ we infer that $c_{ij}=0$, $j\not=i$ and $2c_{jj}=0$. When $L_A=I$, then  $c_{ij}=0$, $j\not=i$ and $2c_{jj}=1$.
\qed
  
\begin{defn}
A square matrix $A\in \mnr$ is called {\it Lyapunov idempotent} or {\it $L$-idempotent}, for short, iff $L_{A}$ is idempotent; that is, $L_A^2=L_A$.
\end{defn}

It is clear that $A$ is $L$-idempotent when $A=0$ and also when $A=\frac{1}{2}I$. Moreover, when $L_A$ is idempotent, then $L_{tA}$ is not, except when $A=0$ or $t\neq 0,1$: Since $L_{tA}=tL_A$ for any $t\in \RR$, we have that $L_{tA}^2=t^2L_A^2=t^2L_A$. Therefore, if $tA$ is $L$-idempotent, then $tL_A=L_{tA}=L_{tA}^2=t^2L_A$, which implies that either $t=0$ or $L_A=tL_A$. When $t\not=0$, the above equality implies that either $t=1$ or $L_A=0$, which is equivalent to: $A=0$. 

\begin{thm}
\label{Ly:ch}
A matrix $A\in \mnr$ is $L$-idempotent iff it is diagonal and  moreover $\sigma(A)\subseteq \big\{0,\frac{1}{2}\big\}$. In particular, the only nonsingular $L$-idempotent matrix is $\frac{1}{2}I$.
\end{thm}
\proof First, $A$ is $L$-idempotent iff for any $X\in \S^n$ we have that 
$$A^2X+2AXA^T+X(A^T)^2=AX+XA^T.$$
In particular, if $A^2=[\hat c_1,\ldots,\hat c_n]$, choosing $X=E_{jj}$, $j=1,\ldots,n$, we obtain, for the left hand side, the expression 
$$[0,\ldots,\hat c_j,\ldots,0]+[0,\ldots,\hat c_j,\ldots,0]^T+2[c_{1j} c_j,\ldots,c_{jj} c_j,\ldots,c_{nj} c_j]^T,$$
whereas, the right hand side reduces to 
$$[0,\ldots,c_j,\ldots,0]+[0,\ldots,c_j,\ldots,0]^T.$$
This implies that $2c_{ij}c_{kj}=0$ for $i,k\not=j$. In particular, $c_{ij}^2=0$ for $i\not=j$ and so $c_j=c_{jj} e_j$, which in turn, implies that $A$ must be diagonal and hence symmetric. Now, let $d_1,\ldots,d_n$ be the diagonal entries of $A$. Finally, by setting $X=I$ in the identity, we conclude that $4A^2=2A$, that is $d_j(2d_j-1)=0$ for any $j=1,\ldots,n$, completing the proof.
\qed

\begin{lem}
\label{Ly:properties}
The Stein operator satisfies the following basic properties:\\
(a) For any $A\in \mathbb{R}^{n \times n}$, $S_A$ is linear.\\
(b) If $P$ is an orthogonal matrix, then,  $S_{PAP^T}(PXP^T)=PS_A(X)P^T.$\\
(c) $S_A$ is the identity operator iff $A=0$ and $S_A$ is the zero operator iff $A=\pm I$.
\end{lem}
\proof We only prove the property (c), since the others are straightforward. For the first part, clearly, $S_A$ is the identity iff $AXA^T=0$ for any symmetric matrix $X$. In particular, taking $X=I$ we get that $AA^T=0$ and so $A=0$.\\
Let us prove the second part. Observe that, if for any $j=1,\ldots,n$, $c_j=(c_{1j},\ldots,c_{nj})^T$ is the $j$-th column of $A$, then 
$$S_A(E_{jj})=[0,\ldots,e_j,\ldots,0]-[c_{1j} c_j,\ldots,c_{jj} c_j,\ldots,c_{nj} c_j]^T.$$
Therefore, when $S_A=0$, this implies that $c_{jj} c_j=e_j$. Thus, $c_{jj}^2=1$ and $c_{ij}=0$, $i\not=j$. 

Finally, when $i<j$, taking into account that $A=[c_{11} e_1,\ldots,c_{nn} e_n]$, we obtain that 
$$0=S_A(E_{ij}) =[0,\ldots,e_j,\ldots,e_i,\ldots,0]-[0,\ldots,c_{ii} c_{jj} e_j,\ldots,c_{jj}c_{ii} e_i\ldots, 0],$$
which implies that $c_{ii}c_{ij}=1$, thereby showing that $A=\pm I$.
\qed

\begin{defn}
A square matrix $A\in \mnr$ is called $S$-idempotent iff  $S_{A}$ is idempotent; that is, $S_A^2=S_A$.
\end{defn}

It is clear that $A$ is $S$-idempotent when $A=0$ and also when $A=\pm I$. 

\begin{thm}
\label{St:ch}
A matrix $A\in \mnr$ is $S$-idempotent iff $A^2=\pm A$. In particular, the only nonsingular $S$-idempotent matrices are $\pm I$.
\end{thm}
\proof First, for any $X\in \S^n(\RR)$ we have that
\begin{eqnarray*}
    S^2_A(X)=S_A(S_A(X))&=&S_A(X)-AS_A(X)A^T\\
    & =& X-AXA^T-A(X-AXA^T)A^T\\
    &=& X-2AXA^T+A^2X(A^T)^2
\end{eqnarray*}
%$$\begin{array}{rl}
%S_A^2(X)=&\hspace{-.25cm}S_A\big(S_A(X)\big)=S_A(X)-AS_A(X)A^T=X-AXA^T-A(X-AXA^T)A^T=X-2AXA^T+A^2X(A^T)^2\end{array}$$
%
and hence $A$ is $S$-idempotent iff 
$$AXA^T=A^2X(A^T)^2$$
for any $X\in \S^n(\RR)$, which in particular implies that $$AE_{jj}A^T=A^2E_{jj}(A^T)^2$$ for any $j=1,\ldots,n$. Therefore, if for any $j=1,\ldots,n$ we consider $c_j=(c_{1j},\ldots,c_{nj})^T$, the $j$-th column of $A$ and  $\hat c_j=(\hat c_{1j},\ldots,\hat c_{nj})^T$, the $j$-th column of $A^2$, then 
$$[c_{1j} c_j,\ldots,c_{nj} c_j]=[\hat c_{1j}\hat c_j,\ldots,\hat c_{nj}\hat c_j] ,\hspace{.25cm} j=1,\ldots,n.
$$

For a fixed $j=1,\ldots,n$, the above equality implies that $\hat c_j=0$ iff $c_j=0$ and moreover $\hat c_j=d_j c_j$, where $d_j\in \RR$ and $d_j\not=0$. Moreover, for any $k=1,\ldots,n$ we also have that 
$c_{kj} c_j=\hat c_{kj}d_j c_j$, so that when $c_j\not=0$ this implies that $c_{kj}=d_j\hat c_{kj}$. Hence $c_j=d_j\hat c_j$. In conclusion, if $c_j\not=0$ we have  $c_j=d_j\hat c_j=d_j^2 c_j$ and hence $d_j^2=1$. Since this statement is also true when $c_j=0$ because then $\hat c_j=0$, we conclude that $A^2=AD$, where $D^2=I$.

Finally, suppose that $d_id_j=-1$ when $1\le i<j\le n$. Since 
$$AE_{ij}A^T=A^2E_{ij}(A^T)^2,\hspace{.15cm}1\le i\le j\le n.$$
we also have 
$$[c_{1j} c_i+c_{1i} c_j,\ldots,c_{nj} c_i+c_{ni} c_j]=[\hat c_{1j}\hat c_i+\hat c_{1i}\hat c_j,\ldots,\hat c_{nj}\hat c_i+\hat c_{ni}\hat c_j],$$
and hence, for any $k=1,\ldots,n$ 
$$c_{kj} c_i+c_{ki} c_j=d_id_j (c_{kj} c_i+c_{ki}\
c_j)=-c_{kj} c_i-c_{ki} c_j$$
which implies that $c_{kj} c_i+c_{ki} c_j=0$, which in turn, leads us to conclude that $c_{kj}c_{ki}=0$. If $c_i\not=0$, then $c_{ki}\not=0$ for some $k=1,\ldots,n$ and hence the above identities imply that $c_{kj}=0$ and hence $ c_j=0$. So, when $d_id_j=-1$ then either $ c_i$ or $c_j$ must be the zero vector. 

Let $i=\min\{k=1,\ldots,n: c_k\not=0\}$. Then, 
$$AD=[d_1 c_1,\ldots,d_i c_i,\ldots,d_nc_n]=d_i[c_1,\ldots, c_i,\ldots,c_n]=d_iA,$$ 
since when $1\le j<i$, $0=c_j=d_j c_j=d_i c_j$, whereas when $j>i$ and $d_j\not=d_i$, then $ c_j=0$. Hence $ 0= c_j=d_j c_j=d_i c_j$. 
\qed

In view of Lemma \ref{idemrangemonotone}, Theorem \ref{Ly:ch} and Theorem \ref{St:ch}, we obtain the following result. We skip the proof.

\begin{thm}\label{idem_result}
Let $A \in \mnr.$ We then have the following:\\
(a) If $A$ is a diagonal matrix, whose diagonal entries are either $0$ or $\frac{1}{2}$, then the Lyapunov operator $L_A$ is range monotone.\\
(c) If $A$ satisfies $A^2=\pm A$, then the Stein operator $S_A$ is range monotone.
\end{thm}

Let us briefly discuss the notion of the group inverse of a linear operator on a finite dimensional vector space $V$ \cite{rob}. Let $T:V \rightarrow V$ be linear. $T$ is said to have a group inverse $L:V \rightarrow V$ if $TLT=T, LTL=L$ and $TL=LT$. 
The group inverse of $T$ need not exist, but when it does exist, it is unique and will be denoted by $T^{\#}.$ We make use of the following in our discussion.

\begin{thm}\cite[Theorem 5]{rob}\label{rob}
Let $T:V \rightarrow V$ be linear. Then the following are equivalent: \\
(a) $T^{\#}$ exists.\\
(b) The subspaces $R(T)$ and $N(T)$, of $V$,  are complementary.\\
(c) $R(T^2)=R(T).$\\
(d) $N(T^2)=N(T).$
\end{thm}

Let us recall that an operator $T$, on an inner product space, is called normal if $TT^{\ast}=T^{\ast}T$, where $T^{\ast}$ denotes the adjoint of $T$. It is easy to see that if $T$ is normal, then $R(T^{\ast})=R(T)$. So, if $T$ is normal, then the subspaces $R(T)$ and $N(T)$ are complementary. Thus, we have the following consequence of Theorem \ref{rob}.

\begin{cor}
Let $T:V \rightarrow V$ be linear, where $V$ is a finite dimensional inner product space. If $T$ is a normal operator, then $T^{\#}$ exists. 
\end{cor}

\begin{defn}
An operator $T:{\cal S}^n \longrightarrow {\cal S}^n$, is called generalized $k$-potent ($k \geq 2$) if there exists a constant $\alpha$ such that $T^k=\alpha T.$ 
\end{defn}

\begin{pro}\label{idemtripo}
Let $T:{\cal S}^n \rightarrow {\cal S}^n$ be a generalized $k$-potent. Then $T^{\#}$ exists. In particular, any idempotent operator is group invertible.
\end{pro}
\proof
We have $T^k=\alpha T$, for some $k \geq 2.$ Let $T^2(x)=0.$ Then $0=T^k(x)=0=\alpha T(x)=0,$ proving that $T(x)=0.$ Thus $N(T^2)=N(T).$ So, the group inverse of $T$ exists. 
\qed

In the next result, we identify four classes of matrices for which the Lyapunov as well as the Stein operators are group invertible. We shall make use of the following observation. Let $A \in \mathbb{R}^{n \times n}$. Then, $L_A^{\ast}=L_{A^T}$ and $S_A^{\ast}=S_{A^T}.$

\begin{thm}\label{gpinvexist}\leavevmode\\
The group inverses $L_A^{\#}$ and $S_A^{\#}$ exist, if $A$ satisfies any of the following conditions:\\
(a) $A^2=-I$.\\
(b) $A^2=I.$\\
(c) $A^T=-A$.\\
(d) $A^T=A.$
\end{thm}
\proof
(a) Let $A^2=-I$. First, we consider the Lyapunov operator $L_A$. We have $$L_A(X)=AX+XA^T, X \in {\cal S}^n.$$ Thus,   
\begin{eqnarray*}
L_{A}^2(X) & = & A^2X + AXA^T + AXA^T+X(A^T)^2 \\
 & = & -2(X-AXA^T). 
 \end{eqnarray*}
So,  
\begin{eqnarray*}
L_A^3(X) & = & -2(L_A(X) - L_A(AXA^T)\\
& = & -2(AX + XA^T - A^2XA^T - AX(A^T)^2) \\
& = & -4(AX+XA^T) \\
& = & -4L_A(X).
\end{eqnarray*} 
Thus, $L_A^3=-4L_A$ and so by Proposition \ref{idemtripo}, $L_A^\#$ exists. \\
Next, we take the case of the Stein operator. We have $$S_A(X)=X-AXA^T, X \in {\cal S}^n.$$ So, 
\begin{eqnarray*}
S_A^2(X) & = & S_A(X-AXA^T) \\
& = & X-AXA^T-AXA^T+A^2X(A^T)^2\\
&=& 2(X-AXA^T)\\
& =& 2S_A(X).
\end{eqnarray*}
So, $S_A^2=S_A$ and again, by Proposition \ref{idemtripo} it follows that $S_A^{\#}$ exists. In fact, in this case, $S_A^{\#}=2S_A$. This proves (a).\\
(b) Let $A^2=I.$ Then, a calculation done as earlier, leads to the formula $$L_{A}^2(X)=2(X+AXA^T).$$ This, in turn, implies that $$L_{A}^3=4L_A.$$ The conclusion on the existence of the group inverse of $L_A$, now follows. \\
For the Stein operator, just as in the case $A^2=-I,$ it follows that $$S_A^2=2S_A$$ and so, $S_A^{\#}$ exists.\\
(c) Since $A^T=-A,$ we have $L_A(X)=AX-XA.$ Now, for $X,Y \in {\cal S}^n,$
\begin{eqnarray*}
 \langle X, L_A(Y)\rangle & = & \rm tr (X(AY-YA))\\
 & = & \rm tr(XAY) - \rm tr (XYA)\\
 & = & -(\rm tr (XYA) - \rm tr (XAY))\\
 & = & -(\rm tr (AXY) - \rm tr (XAY))\\
 & = & -(\rm tr((AX-XA)Y))\\
 & = & -\langle L_A(X), Y\rangle.
\end{eqnarray*}
Hence $L_A^*=-L_{A},$ which means that $R(L_A)=R(L_A^{\ast})$. Thus $R(L_A)$ and $N(L_A)$ are complementary subspaces. Thus, $L_A^\#$ exists. 

Next, let us consider the Stein operator. Since $A^T=-A,$ we have $S_A(X)=X+AXA.$ Now, 
\begin{eqnarray*}
\langle X, S_A(Y)\rangle & = & \rm tr(X(Y+AYA))\\
& = & \rm tr(XY+XAYA) \\
&=& \rm tr(XY) + \rm tr(AXAY) \\
&=& \rm tr((X+AXA)Y)\\
&=& \langle S_A(X), Y\rangle.
\end{eqnarray*}
Hence $S_A^*=S_{A}.$ Once again, the subspaces $R(S_A)$ and $N(S_A)$ are complementary and so $S_A^\#$ exists. \\
(d) When $A=A^T,$ as was remarked earlier, we have $L_A^*=L_{A^T}=L_A$ and $S_A^*=S_{A^T}=S_A.$ Thus, $L_A^\#$ and $S_A^\#$ exist.
\qed

Let $A \in \mnr$. We say that $A$ is {\it nonnegative stable}, if all the eigenvalues of $A$ have a nonnegative real part. $A$ will be referred to as {\it Shur semi-stable}, if all its eigenvalues lie in the closed unit disc of the complex plane. 

\begin{rem}\label{remlyapunov}
Some remarks are in order. Let matrix $A$ be nilpotent (so that $A$ is nonnegative stable). Then it is easy to show that the operator $L_A$ is also nilpotent. As noted earlier, it follows that the group inverse of $L_A$ does not exist. So, $L_A$ is not trivially range monotone. Nevertheless, we identify a class of matrices $A$, which are nonnegative stable, for which the associated Lyapunov operator is trivially range monotone. Example \ref{remstein} shows that not all Schur semi-stable matrices are trivially range monotone. However, we identify a class of Schur semi-stable matrices $A$, for which the Stein operator is trivially range monotone. Both these are presented in the next result. 
\end{rem}

In what follows, we consider the question of trivial range monotonicity of the Lyapunov and the Stein operators corresponding to matrices that satisfy one of the first three conditions of Theorem \ref{gpinvexist}.

\begin{thm}\label{trangemonotone}
    For the Lyapunov operator $L_A$ and Stein operator $S_A$ the following are true:\\
    $(a)$ Let $A^2=-I.$ Then $L_A$ and $S_A$ are trivially range monotone.\\
    $(b)$ Let $A^2=I.$ Then $S_A$ is trivially range monotone.\\
    $(c)$ Let $A^T=-A.$ then $L_A$ is trivially range monotone. 
    \end{thm}
\proof
First, we observe that both the Lyapunov and the Stein operators are group invertible, under any of the three conditions on the matrix $A$, as above (in view of Theorem \ref{gpinvexist}). Thus in all these cases, the range space and the null space (of either the Lyapunov operator or the Stein operator) are complementary subspaces of ${\cal S}^n.$ We shall make repeated use of this fact, in our proofs.

    $(a)$ Let $L_A(X)\succeq 0$ with $X\in R(L_A).$ Then $0 \preceq L_A^3(X) =-4L_A(X).$ This means that $L_A(X)=0.$ Thus, $X\in R(L_A) \cap N(L_A)=\{0\}$, proving that $L_A$ is trivially range monotone. 
    
Next, we show the trivial range monotonicity of the Stein operator. Let $X\in R(S_A)$ so that there exists $Y \in \mathcal{S}^n$ with $S_{A}(Y)=X$. Also, let $S_A(X)\succeq 0.$ Then, from the calculation as earlier, we have  
\begin{eqnarray*}
0 \preceq S_A(X)=S_A^2(Y)=2S_A(Y)=2(Y-AYA^T).
\end{eqnarray*}
Set $Z:=Y-AYA^T.$ Then, $Z \succeq 0$. This means that $$0 \preceq AZA^T= AYA^T-A^2Y(A^T)^2=AYA^T-Y=-Z.$$
Thus, $Z=0$, i.e. $S_A(X)=0,$ which in turn implies that that $X=0,$ as $X\in R(S_A) \cap N(S_A).$

$(b)$ The proof of the trivial range monotonicity of the Stein operator, is entirely similar to the second part of the above result and is skipped. 

$(c)$ Let $Z:=L_A(X)=AX-XA \succeq 0,$ with $X\in R(L_A).$ Then all the eigenvalues of $Z$ are non-negative. Further, $\rm tr (Z)=\rm tr(AX-XA)=0$ and so, all the eigenvalues of $Z$ are zero. Since $Z$ is diagonalizable, $0=Z=L_A(X).$ Therefore $X=0,$ as $X\in R(L_A) \cap N(L_A).$ This shows the trivial range monotonicity property of $L_A$. \\
\qed

Let matrix $A$ be such that either $A^2=I$ or $A^2=-I.$ Then $A^{-1}$ exists and shares the same property as that of $A$. If $A$ is skew-symmetric, then $A^{\#}$ exists (since $R(A)$ and $N(A)$ are complementary) and $A^{\#}$ is also skew-symmetric. Thus we have the following immediate consequence of Theorem \ref{trangemonotone}.

\begin{cor}
Let $A^2=-I.$ Then $L_{A^{-1}}$ and $S_{A^{-1}}$ are trivially range monotone. If $A^2=I,$ then $S_{A^{-1}}$ is trivially range monotone. Let $A^T=-A.$ Then $L_{A^{\#}}$ is trivially range monotone. 
\end{cor}

The first example below illustrates (a) of Theorem \ref{trangemonotone}.

\begin{ex}\label{illus_st(a)}
Let $A=
\begin{pmatrix}
~~0 & 1 \\
-1 & 0
\end{pmatrix},$ so that $A^2=-I.$ For any $X=
\begin{pmatrix}
a & b \\
b & c
\end{pmatrix} \in {\cal S}^2,$ we have $$S_A(X)=
\begin{pmatrix}
a-c & 2b \\
2b & c-a
\end{pmatrix}.$$ If $S_A(X) \succeq 0,$ then $a=c$ and so $b=0$. Thus, 
$X=\begin{pmatrix}
a & 0 \\
0 & a
\end{pmatrix}.$ If we impose the condition that $X \in R(S_A),$ then $a= -a$, proving that $X=0.$ Thus, $S_A$ is trivially range monotone.
\end{ex}

The next example illustrates item (b) of Theorem \ref{trangemonotone}.

\begin{ex}\label{illus_st(b)}
Let $A=
\begin{pmatrix}
0 & 1 \\
1 & 0
\end{pmatrix},$ so that $A^2=I.$ For any $X=
\begin{pmatrix}
a & b \\
b & c
\end{pmatrix} \in {\cal S}^2,$ we have $$S_A(X)=
\begin{pmatrix}
a-c & 0 \\
0 & c-a
\end{pmatrix}.$$ If $S_A(X) \succeq 0,$ then $a=c$ and so 
$X=\begin{pmatrix}
a & b \\
b & a
\end{pmatrix}.$ Further, if $X \in R(S_A),$ then $a= -a$ and $b=0$, proving that $X=0.$ Thus, $S_A$ is trivially range monotone.
\end{ex}

Items (a) (for the Lyapunov operator) and (c) of Theorem \ref{trangemonotone} are illustrated, next.

\begin{ex}\label{illus_lyst}
Let $A=
\begin{pmatrix}
~~0 & 1 \\
-1 & 0
\end{pmatrix}.$ Then $A$ is skew-symmetric and $A^2=-I.$ For any $X=
\begin{pmatrix}
a & b \\
b & c
\end{pmatrix} \in {\cal S}^2,$ we have $$L_A(X)=
\begin{pmatrix}
2b & c-a \\
c-a & -2b
\end{pmatrix}.$$ The requirement that $L_A(X) \succeq 0$ yields $b=0$ and so $a=c$. Thus $X=a I$. The condition that $X \in R(L_A)$ then implies that $a=0$, so that $X=0,$ proving the trivial range monotonicity of $L_A.$ 
\end{ex}

\section{Counterexamples}\label{cex}
The Lyapunov operator, is not trivially range monotone, in general, when $A$ is an involutory matrix. 

\begin{ex}\label{invlyo}
Let $A=
\begin{pmatrix}
0 & 1 \\
1 & 0
\end{pmatrix},$ so that $A^2=I$ and for any $X=
\begin{pmatrix}
a & b \\
b & c
\end{pmatrix} \in {\cal S}^2,$ we have $$L_A(X)=
\begin{pmatrix}
2b & a+c \\
a+c & 2b
\end{pmatrix}.$$ Note that, if $X=\begin{pmatrix}
1 & 0 \\
0 & -1
\end{pmatrix},$ then $L_A(X)=0,$ proving that $L_A$ is singular. Next, take $Y=A$ and $U=
\begin{pmatrix}
1 & 0 \\
0 & 0
\end{pmatrix}.$ Then $L_A(U)=Y$ so that $Y \in R(L_A).$ We have $L_A(Y)=2I \succeq 0$, but $Y \nsucceq 0.$ Thus, the Lyapunov operator is not even range monotone, if $A$ is involutory. 
\end{ex}

The conclusion above holds for matrices of higher order, too. For instance, 
if $A$ is involutory and if one sets  $\tilde{A}
=\begin{pmatrix}
A & 0\\
0 & \pm 1
\end{pmatrix},$ then $L_{\tilde{A}}$ not range monotone. We omit the details.

\begin{ex}\label{skewssteinorder2}
Consider the case, when $A^T=-A.$ For $n=1, S_A(X)=X.$ Here, $S_A$ is even monotone. For $n=2,$ (after normalizing) $A$ must be of the following form 
$$\begin{pmatrix}
~0 & \pm 1 \\
\mp 1 & ~0
\end{pmatrix}.$$
Then, for any symmetric matrix $X=\begin{pmatrix}
a & b \\
b & d
\end{pmatrix},$
we have 
$$ S_A(X)=
\begin{pmatrix}
a-d & 2b \\
2b & d-a
\end{pmatrix}.
$$
$S_A$ is singular since $S_A(I)=0.$ Next, if $S_A(X) \succeq 0,$, then $a=d$ and so $b=0$. Further, if $X \in R(S_A),$ then we have $a=-d$, and so $a=d=0,$ proving that $X=0$. This proves the trivial range monotonicity of $S_A,$ for $n=2.$
\end{ex}

In the next two examples, we show that the Stein operator is not trivially range monotone for $n=3,4$, for a skew-symmetric matrix $A$.

\begin{ex}\label{skewsstein1}
Let $A=\frac{1}{\sqrt 2}
\begin{pmatrix}
~0 & ~1 & ~0 \\
-1 & ~0 & ~1 \\
~0 & -1 & ~0 
\end{pmatrix},$ so that for any symmetric matrix $X=
\begin{pmatrix}
a & b & c\\
b & d & e \\
c & e & f 
\end{pmatrix},$ we have 
$$S_A(X)= \frac{1}{2}\begin{pmatrix}
2a-d & \ 3b-e & 2c+d \\
3b-e & 2d+2c-a-f & 3e-2b \\
2c+d & 3e-2b & 2f-d
\end{pmatrix}.$$
$S_A$ is singular, since $S_A(X)=0,$ for $X= 
\begin{pmatrix}
-1 & ~0 & ~1\\
~0 & -2 & ~0\\
~1 & ~0 & -1
\end{pmatrix}.$ Next, if $U=
\begin{pmatrix}
2 & 0 & 1 \\
0 & 1 & 0 \\
1 & 0 & 2
\end{pmatrix}$ and $Y=\frac{1}{2}
\begin{pmatrix}
3 & 0 & 3 \\
0 & 0 & 0 \\
3 & 0 & 3
\end{pmatrix}$ then $S_A(U)=Y$ so that $Y \in \ R(S_A).$ We have 
$S_A(Y)=\frac{1}{2}
\begin{pmatrix}
3 & 0 & 3 \\
0 & 0 & 0 \\
3 & 0 & 3
\end{pmatrix}\succeq 0$ and $Y\succeq 0.$ Since $Y \neq 0$, we conclude that the Stein operator is not trivially range monotone.
\end{ex}

\begin{ex}\label{skewsstein2}
Let $A=
\begin{pmatrix}
~0 & 1 & ~0 & 0 \\
-1 & 0 & ~0 & 0 \\
~0 & 0 & ~0 & 2 \\
~0 & 0 & -2 & 0
\end{pmatrix},$ so that for any symmetric matrix $X=
\begin{pmatrix}
a & b & c & d \\
b & e & f & g \\
c & f & h & i \\
d & g & i & j
\end{pmatrix},$ we have $$S_A(X)=
\begin{pmatrix}
a-e & 2b & c-2g & d+2f \\
2b & e-a & f+2d & g-2c \\
c-2g & f+2d & h-4j & 5i \\
d+2f & g-2c & 5i & j-4h
\end{pmatrix}.$$
It is clear that, if $X=
\begin{pmatrix}
1 & 0 & 0 & 0 \\
0 & 1 & 0 & 0 \\
0 & 0 & 0 & 0 \\
0 & 0 & 0 & 0
\end{pmatrix},$ then $S_A(X)=0,$ proving that $S_A$ is singular. Next, if $Y=
\begin{pmatrix}
0 & 0 & ~0 & ~0 \\
0 & 0 & ~0 & ~0 \\
0 & 0 & -3 & ~0 \\
0 & 0 & ~0 & -3
\end{pmatrix}$ and $U=
\begin{pmatrix}
0 & 0 & 0 & 0 \\
0 & 0 & 0 & 0 \\
0 & 0 & 1 & 0 \\
0 & 0 & 0 & 1
\end{pmatrix},$ then $S_A(U)=Y$ so that $Y \in R(S_A).$ We have $S_A(Y)=
\begin{pmatrix}
0 & 0 & ~0 & ~0 \\
0 & 0 & ~0 & ~0 \\
0 & 0 & ~9 & ~0 \\
0 & 0 & ~0 & ~9
\end{pmatrix} \succeq 0.$ However, $Y \nsucceq 0,$ proving that the Stein operator is not even range monotone.
\end{ex}

\begin{ex}\label{skewsstein3}
For $n = 5,$ consider the skew symmetric block diagonal matrix $B,$ whose leading principal sub-block is $A,$ as in example \ref{skewsstein1} and whose trailing principal sub-block matrix is 
$\begin{pmatrix}
 0 & -1 \\
 1 & ~0
\end{pmatrix}.$ Then, by the argument given as earlier, it may be shown that $S_B$ is not trivially range monotone. This inductive argument allows us to conclude that, {\it in general},  $S_A$ is not trivially range monotone when $A$ is a skew symmetric matrix, of any order $n \geq 3$.   
\end{ex}

This finishes the discussion for skew-symmetric case.  

Next, let $A$ be symmetric. Example \ref{invlyo}, shows that $L_A$ is not trivially range monotone. The following example shows that $S_A$ is not trivially range monotone.

\begin{ex}\label{symstein}
Consider the symmetric matrix $A=
\begin{pmatrix}
1 & 0 \\
0 & 2
\end{pmatrix},$ so that for any $X=
\begin{pmatrix}
a & b \\
b & c
\end{pmatrix}\in \mathcal{S}^2,$ $S_A$ is given by 
$$S_A(X)=
\begin{pmatrix}
~0 & -b \\
-b & -3c
\end{pmatrix}.$$ $S_A$ is not invertible, as $S_A(X)=0,$ for $X=
\begin{pmatrix}
1 & 0\\
0 & 0
\end{pmatrix}.$ Note that, by item $(d)$ of Proposition \ref{gpinvexist}, $S_A^{\#}$ exists. Next, if $Y=
 \begin{pmatrix}
 0 & ~0 \\
 0 & -1
 \end{pmatrix}$ then $S_A(U)=Y$, where $U=
\begin{pmatrix}
0 & 0 \\
0 & \frac{1}{3}
\end{pmatrix},$ showing that $Y \in R(S_A).$ Also, $Y \nsucceq 0,$ whereas, $S_A(Y)=
\begin{pmatrix}
0 & 0 \\
0 & 3
\end{pmatrix}\succeq 0$. 
\end{ex}

We summarize the findings in the following table. Note that, while "Yes" stands for an affirmative answer, "No" means that counterexamples exist to show that the answers are negative, in general.

\begin{center}
\begin{tabular}{|c|c|c|}
\hline
{\bf Matrix classes}  & \multicolumn{2}{c|}{\bf Trivial Range Monotonicity} \\

\cline{2-3}  & $L_A$  & $S_A$  \\
\hline $A^2=-I$ & Yes [Theorem \ref{trangemonotone}(a)] & Yes [Theorem \ref{trangemonotone}(a)]\\
 
\hline $A^2=I$ & No [Example \ref{invlyo}] & Yes [Theorem \ref{trangemonotone}(b)]\\

\hline $A^T=-A$ & Yes [Theorem \ref{trangemonotone}(c)] & Yes ($n=2$) [Example \ref{skewssteinorder2}] \\
& & No ($n\geq 3$) [Examples \ref{skewsstein1}, \ref{skewsstein2}, \ref{skewsstein3}]\\
 
% \hline  &&SMD1:\\
%  Six monthly DC meeting& After five years from the date of&SMD2:\\
%  &registration, upto maximum&SMD3:\\
%  &period of the programme&SMD4:\\
\hline $A^T=A$  & No [Example \ref{invlyo}] & No [Example \ref{symstein}]\\
\hline
\end{tabular}
\end{center}

\section{Concluding Remarks}

We have studied four matrix classes in the context of the notion of trivial range monotonicity of the associated Lyapunov and the Stein operators. The motivation, as was mentioned in the introduction, is to obtain a generalization of the trivial range monotonicity property of singular irredudcible $M$-matrices. It would be interesting to bring the matrix classes for which we could obtain affirmative results, under a more general framework. Similarly, to unify the matrix classes for which negative results for operator analogues, have been proved. It would be another interesting question to study the notion of irreducibility of the Lyapunov operator and the Stein operator. It is pertinent to point to the fact that there are different (possibly nonequivalent) ways of defining irreducibility of a "positive" operator. So one could propose such a notion for the Stein operator (which is the difference: $I$ minus a positive operator $X \rightarrow AXA^T$), simply by introducing the assumption of irreducibility for the second term, which is positive (meaning, that the matrix $AXA^T$ is symmetric and positive semidefinite, whenever so is $X$). However, it is not clear how one could propose irreducibility for the Lyapunov operator and the likes of it. While we will pursue these questions in future, we note that, motivated by the considerations of this article, the notion of irreducibility has been investigated for $Z$-operators over Euclidean Jordan Algebras \cite{gowdanew}.

\section{Acknowledgements}
Samir Mondal acknowledges funding received from the Prime Minister's Research Fellowship (PMRF), Ministry of Education, Government of India, for carrying out this work. The first and the third authors thank ASEM-DUO for financial support enabling the former's visit to India and the latter, to Spain. The first author was also partially supported by the Spanish I+D+i program under project PID2021-122501NB-I00. The authors thank Prof. M.S. Gowda for his suggestions and comments that have helped in a clearer presentation of the material. He suggested the nomenclature "positive stable $Z$-operators".

\end{document}